\documentclass[11pt]{amsart}
\usepackage{latexsym,amssymb,amsmath,epsfig}
\usepackage{graphicx}
\newtheorem{theorem}{Theorem}[section]
\theoremstyle{plain}

\newtheorem{corollary}{Corollary}[section]

\newtheorem{lemma}{Lemma}[section]

\newtheorem{proposition}{Proposition}[section]
\newtheorem{remark}{Remark}[section]

\numberwithin{equation}{section}

\newtheorem{openQ}{Open question}
\begin{document}

\title{Perfect (super) edge-magic crowns}

\author{S. C. L\'opez}
\address{%
Departament de Matem\`{a}tiques\\
Universitat Polit\`{e}cnica de Catalunya. BarcelonaTech\\
C/Esteve Terrades 5\\
08860 Castelldefels, Spain}
\email{susana.clara.lopez@upc.edu}

\author{F. A. Muntaner-Batle}
\address{Graph Theory and Applications Research Group \\
 School of Electrical Engineering and Computer Science\\
Faculty of Engineering and Built Environment\\
The University of Newcastle\\
NSW 2308
Australia}
\email{famb1es@yahoo.es}

\author{M. Prabu}
\address{British University Vietnam\\
Hanoi, Vietnam}
\email{mprabu201@gmail.com}

\maketitle

\begin{abstract}
In this paper we continue the study of the valences for (super) edge-magic labelings of crowns $C_{m}\odot \overline K_{n}$ and we prove that the crowns are perfect (super) edge-magic when $m=pq$ where $p$ and $q$ are different odd primes. We also provide a lower bound for the number of different valences of $C_{m}\odot \overline K_{n}$, in terms of the prime factors of $m$.
\end{abstract}

\begin{quotation}

\end{quotation}


\thispagestyle{empty}

\section{Introduction}
For the graph theory terminology and notation not defined in this paper we refer the reader to either one of the following sources \cite{BaMi,CH,G,Wa}.
However, in order to make this paper reasonably self-contained, we mention that by a $(p,q)$-graph we mean a graph of order $p$ and size $q$.
In 1970, Kotzig and Rosa \cite{K1} introduced the concepts of edge-magic graphs and edge-magic labelings as follows: Let $G$ be a $(p,q)$-graph. Then $G$ is called {\it edge-magic} if there is a bijective function $f:V(G)\cup E(G)\rightarrow \{i\}_{i=1}^{p+q}$ such that the sum $f(x)+f(xy)+f(y)=k$ for any $xy\in E(G)$. Such a function is called an {\it edge-magic labeling} of $G$ and $k$ is called the {\it valence} \cite{K1} or the {\it magic sum} \cite{Wa} of the labeling $f$. We write $val(f)$ to denote the valence of $f$.

Motivated by the concept of edge-magic labelings, Enomoto et al. \cite{E} introduced in 1998 the concepts of super edge-magic graphs and labelings as follows: Let $f:V(G)\cup E(G) \rightarrow\{i\}_{i=1}^{p+q}$ be an edge-magic labeling of a $(p,q)$-graph G with the extra property that $f(v)=\{i\}_{i=1}^{p}.$ Then G is called { \it super edge-magic} and $f$ is a {\it super edge-magic labeling} of $G$. It is worthwhile mentioning that Acharya and Hegde had already defined in \cite{AH} the concept of strongly indexable graph that turns out to be equivalent to the concept of super edge-magic graph. We take this opportunity to mention that although the original definitions of (super) edge-magic graphs and labelings were originally provided for simple graphs (that is to say, graphs with no loops nor multiple edges), in this paper, we understand these definitions for any graph. Therefore, unless otherwise specified, the graphs considered in this paper are not necessarily simple. In \cite{F2}, Figueroa-Centeno et al. provided the following useful characterization of super edge-magic simple graphs, that works in exactly the same way for graphs in general.

\begin{lemma}\label{super_consecutive} \cite{F2}
Let $G$ be a $(p,q)$-graph. Then $G$ is super edge-magic if and only if  there is a
bijective function $g:V(G)\longrightarrow \{i\}_{i=1}^p$ such
that the set $S=\{g(u)+g(v):uv\in E(G)\}$ is a set of $q$
consecutive integers.
In this case, $g$ can be extended to a super edge-magic labeling $f$ with valence $p+q+\min S$.
\end{lemma}

Enomoto et al. \cite{E} were the first ones to observe the following result for which we provide the proof as a matter of completeness.

\begin{lemma}\cite{E}\label{CnSEM}
A cycle of order $n$ is super edge-magic when $n$ is odd.
\end{lemma}

\begin{proof}
 Let $V(C_{n})=\{ v_i \}_{i=1}^{n}$ and $E(C_n)= \{v_iv_{i+1} \}_{i=1}^{n-1} \cup \{v_nv_1 \}$. The function $f:V(C_{n})\rightarrow \{i \}_{i=1}^{n}$ defined by the rule
 $$f(v_i)=\left\{\begin{array}{ll}
                   (i+1)/2, & \hbox{if} \ i \ \hbox{is odd} \\
                 (i+1+n)/2, & \hbox{if} \ i\ \hbox{is even}
                 \end{array}\right.$$
is a super edge-magic labeling of $C_{n}$.
\end{proof}

\begin{figure}
\begin{center}
  \includegraphics [width=80pt]{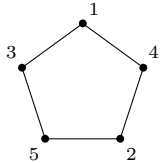}
  \end{center}
  \caption{A super edge-magic labeling of $C_{5}$.}\label{Fig1}
\end{figure}

We will refer to the labeling introduced in the proof of the previous lemma as the {\it canonical} labeling of the cycle. When we say that a digraph has a labeling we mean that its underlying graph has such labeling, see \cite{F1}. We denote the underlying graph of a digraph $D$ by $und(D)$.

\begin{remark}\label{strongorientationofcm}
Let $\{f(C_n)^+, f(C_n )^-\}$ be the strong orientations of the super edge-magic labeled cycle $C_n$ introduced in the proof of Lemma \ref{CnSEM}. Then
\begin{eqnarray} \label{conditionforcycle1}
(a,b) \in E(f(C_n)^+) \Leftrightarrow b-a &\equiv& \frac{n+1}{2} \ (mod \ n)
\end{eqnarray}
\begin{eqnarray} \label{conditionforcycle2}
(a,b) \in E(f(C_n)^-) \Leftrightarrow b-a &\equiv& \frac{n-1}{2} \ (mod \ n)
\end{eqnarray}
\end{remark}

Let $G=(V,E)$ be a $(p,q)$-graph, and denote by
$T_G$ the set
$$\left\{\frac{\sum_{u\in V}\mbox{deg}(u)g(u)+\sum_{e\in E}g(e)}q:\ g:V\cup E \rightarrow \{i\}_{i=1}^{p+q} \ \mbox{ is a bijective function}\right\}.$$
 If $\lceil\min T_G\rceil\le  \lfloor\max T_G\rfloor$ then the {\it magic interval} of $G$, denoted by $J_G$, is defined to be the set
$J_G=\left[\lceil\min T_G\rceil, \lfloor\max T_G\rfloor\right]\cap \mathbb{Z}$
and the {\it magic set} of $G$, denoted by $\tau_G$, is the set
$\tau_G=\{n\in J_G:\ n \ \mbox{is the valence of some edge-magic labeling of}\ G\}.$
It is clear that $\tau_G\subseteq J_G$. A graph $G$ is called {\it perfect edge-magic } \cite{PEM_LMR} if $\tau_G=J_G$.

A famous conjecture of Godbold and Slater \cite{GodSla98} states that, for $n=2t+1\ge 7$ and $5t+4\le j\le 7t+5$ and for $n=2t\ge 4$ and $5t+2\le j\le 7t+1$ there is an edge-magic labeling of $C_n$, with valence $k=j$. That is, for odd $n\ge 7$ and for even $n\ge 4$ the cycle $C_n$ is perfect edge-magic.

Let $G=(V,E)$ be a $(p,q)$-graph. Then the set $S_{G}$ is defined as $S_{G}= \{ 1/q( \Sigma_{u \in V}  deg(u)g(u)+ \Sigma_{i=p+1}^{p+q} i ):$ the function $g:V \rightarrow \{i\}_{i=1}^{p}$ is bijective\}. If $\lceil\min S_G\rceil\le  \lfloor\max S_G\rfloor$ then the {\it super edge-magic interval} of $G$, denoted by $I_G$, is defined to be the set
$I_G=\left[\lceil\min S_G\rceil, \lfloor\max S_G\rfloor\right]\cap \mathbb{Z}$
and the {\it super edge-magic set} of $G$, denoted by $\sigma_G$, is the set formed by all integers
$k\in I_G$ such that $k$ is the valence of some super edge-magic labeling of $G$.
A graph $G$ is called {\it perfect super edge-magic graph} \cite{LopMunRiu5} if $\sigma_G=I_G$.

\begin{lemma} \label{k1nsem}
The graph formed by a star $K_{1,n}$ and a loop attached to its central vertex, denoted by $K_{1,n}^{l}$, is perfect super edge-magic for all positive integers $n$. Furthermore, $|I_{K_{1,n}^{l}}|=|\sigma_{K_{1,n}^{l}}|=n+1$.
\end{lemma}

\begin{proof}
By Lemma \ref{super_consecutive}, it is a very easy observation that any bijection $f:V(K_{1,n}^{l}) \rightarrow \{j \}_{j=1}^{n+1}$ is a super edge-magic labeling of ${K_{1,n}^{l}}$. Further more, the valence of any super edge-magic labeling of $K_{1,n}^{l}$ depends only on the label assigned to the central vertex of $K_{1,n}^{l}$ (that is, the vertex of $K_{1,n}^{l}$ with degree different from 1). If two labelings of $K_{1,n}^{l}$ assign consecutive labels to the central vertex of $K_{1,n}^{l}$, then the resulting valences are also consecutive. Since there are exactly $(n+1)$ possible consecutive labels to assign to the central vertex, it follows that $|I_{K_{1,n}^{l}}|=|\sigma_{K_{1,n}^{l}}|=n+1$.
\end{proof}

\begin{figure}
\begin{center}
\includegraphics [width=215pt]{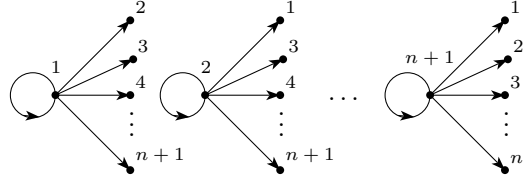}\\
\end{center}
  \caption{All possible super edge-magic labelings of an orientation of $K_{1,n}^{l}.$}\label{Fig2}
\end{figure}

The corona product of two graphs $G$ and $H$ is the graph $G \odot H$ obtained by placing a copy of $G$ and $|V(G)|$ copies of $H$ and then joining each vertex of $G$ with all vertices in one copy of $H$ in such a way that all vertices in the same copy of $H$ are joined exactly to one vertex of $G$. Let $\overline{K}_n$ be the complementary graph of the complete graph $K_n$, $n\in \mathbb{N}$.

\begin{theorem}\cite{LopMunRiu5,PEM_LMR} \label{powerofprime}
Let $C_m$ be a cycle of order $m=p^k$, where $p>2$ is a prime number. Then the graph $G \cong C_m \odot \overline{K}_n$ is perfect (super) edge-magic.
\end{theorem}

In this paper, we extend the result to $m=pq$, where $p$ and $q$ are different odd primes.  The paper is organized as follows: in Section 2, we provide all the necessary results needed for this paper. In Section 3, we prove that each element in the family $C_m\odot\overline{K}_n$ where $m=pq$, with  $p$ and $q$ being different odd primes is a perfect (super) edge-magic graph. In Section 4, we provide a lower bound for the number of valences of general crowns $C_m\odot\overline{K}_n$.

\section{The tools}

Let $f$ be an edge-magic labeling of a $(p,q)$-graph $G$. The {\it complementary labeling} of $f$, denoted by $\overline{f}$, is the labeling defined by the rule: $\overline{f}(x)=p+q+1-f(x)$, for all $x\in V(G)\cup E(G)$. Notice that, if $f$ is an edge-magic labeling of $G$, we have that $\overline{f}$ is also an edge-magic labeling of $G$ with valence $\hbox{val}(\overline{f}) = 3(p+q+1)-\hbox{val}(f)$. In the case of a super edge-magic labeling $f$ of a graph $G$, there is also the corresponding \textit{super edge-magic complementary labeling}, $f^{c}$, which is also super edge-magic. In this case $f^{c}$ is defined by the rule $ f^{c}(x) = p+1-f(x), \ \forall x \in V(G)$ and
  $f^{c} (ab)$ is obtained as described in Lemma \ref{super_consecutive}, for all $ab \in E(G)$. Then, the valence of $f^{c}$ can be expressed in terms of the valence of $f$ as follows:
  \begin{eqnarray}\label{formula: valence of the super complentary}
  \hbox{val}(f^{c})&=&4p+q+3-\hbox{val}(f).
  \end{eqnarray}

The complementary labeling of an edge-magic labeling is a powerful tool that allows us to increase the number of valences of certain families of graphs dramatically. Using the complementary labeling we may even prove the perfect edge-magicness of many graphs. The following proposition can serve as an illustration of this fact.

\begin{proposition}\label{propo: K_1_n_PEM}
The graph $K_{1,n}^l$ is perfect edge-magic  for all positive integers. Furthermore, $|J_{K_{1,n}^{l}}|=|\tau_{K_{1,n}^{l}}|=2n+2$.
\end{proposition}

\begin{proof}
Easy calculations show that $J_{K_{1,n}^{l}}=[2n+4,4n+5]\cap \mathbb{Z}$. In Lemma \ref{k1nsem} it is shown that all numbers in the set $\{2n+4,2n+5,\ldots, 3n+4\}$ are in $\tau_{K_{1,n}^{l}}$. Now, using the fact that if $f$ is a super edge-magic labeling of $K_{1,n}^l$ then $\hbox{val}(\overline f)=3(2n+3)-\hbox{val}(f)$, we obtain that $\{3n+5,3n+6,\ldots, 4n+5\}\subseteq\tau_{K_{1,n}^{l}}$, showing the result.
\end{proof}

Also, in the case that $G$ is a graph of equal order and size, new edge-magic labelings can be obtained from known super edge-magic labelings of $G$.
  The \textit{odd labeling} and the \textit{even labeling} \cite{PEM_LMR} obtained from $f$, denoted respectively by $o(f)$ and $e(f)$, are the labelings $o(f), e(f):V(G)\cup E(G)\rightarrow \{i\}_{i=1}^{p+q}$ defined as follows: (i) on the vertices: $o(f)(x)=2f(x)-1$ and $e(f)(x)=2f(x)$, for all $x \in V(G)$, (ii) on the edges: $o(f)(xy) = 2 \hbox{val}(f)-2p-2-o(f)(x)-o(f)(y)$ and $e(f)(xy)=2 \hbox{val}(f)-2p-1-e(f)(x)-e(f)(y)$, for all $xy \in E(G)$.

\begin{lemma} \cite{PEM_LMR}\label{lemma: oddevenlabelings}
Let G be a $(p,q)$-graph with $p=q$ and let $f:V(G)\cup E(G)\rightarrow \{i\}_{i=1}^{p+q}$ be a super edge-magic labeling of $G$. Then, the odd labeling $o(f)$ and the even labeling $e(f)$ obtained from $f$ are edge-magic labelings of G with valences $\hbox{val}(o(f))=2\hbox{val}(f)-2p-2$ and $\hbox{val}(e(f))=2\hbox{val}(f)-2p-1$ respectively.
\end{lemma}

At this point, we want to observe that Proposition \ref{propo: K_1_n_PEM} can also be proved using the labelings provided in the proof of Lemma \ref{k1nsem} and the odd and even labelings just defined above.

In \cite{F1}, Figueroa et al. defined the following product: Let $D$ be a digraph and let $\Gamma$ be a family of digraphs with the same set $V$ of vertices. Assume that $h: E(D) \to \Gamma$ is any function that assigns elements of $\Gamma$ to the arcs of $D$. Then the digraph $D \otimes _{h} \Gamma $ is defined by (i) $V(D \otimes _{h} \Gamma)= V(D) \times V$ and (ii) $((a,i),(b,j)) \in E(D \otimes _{h} \Gamma) \Leftrightarrow (a,b) \in E(D)$ and $(i,j) \in E(h(a,b))$. Note that when $h$ is constant, $D \otimes _{h} \Gamma$ is the Kronecker product. Many relations among labelings have been established using the $\otimes_h$-product and some particular families of graphs, namely $\mathcal{S}_p$ and $\mathcal{S}_p^k$ (see for instance, \cite{ILMR,LopMunRiu1,LopMunRiu6,PEM_LMR2}).
The family $\mathcal{S}_p$ contains all super edge-magic $1$-regular labeled digraphs of order $p$ where each vertex takes the name of the label that has been assigned to it. A super edge-magic digraph $F$ is in $\mathcal{S}_p^k$ if $|V(F)|= |E(F)|=p$ and the minimum sum of the labels of the adjacent vertices is equal to $k$ (see Lemma \ref{super_consecutive}). Notice that, since each $1$-regular digraph has minimum edge induced sum equal to $(p+3)/2$, it follows that $ \mathcal{S}_p \subset \mathcal{S}_p^{(p+3)/2}$. The following result was introduced in \cite{LopMunRiu6}, generalizing a previous result found in \cite{F1} :

\begin{theorem} \label{spk} \cite{LopMunRiu6}
Let $D$ be a (super) edge-magic digraph and let $h: E(D) \to \mathcal{S}_p^k$ be any function. Then $D\otimes _{h} \mathcal{S}_p^k$ is (super) edge-magic.
\end{theorem}

\begin{remark}\label{remarkspk}
The key point in the proof of Theorem \ref{spk} is to rename the vertices of $D$ and each element of $\mathcal{S}_p^k$ after the labels of their corresponding (super) edge-magic labeling $f$ and their super edge-magic labelings respectively. Then the labels of the product are defined as follows: (i) the vertex $(a,i) \in V(D\otimes _{h} \mathcal{S}_p^k)$ receives the label: $p(a-1)+i$ and (ii) the arc $((a,i),(b,j)) \in E(D\otimes _{h} \mathcal{S}_p^k)$ receives the label: $p(e-1)+(k+p)-(i+j)$, where $e$ is the label of $(a,b)$ in D. Thus, for each arc $((a,i),(b,j)) \in E(D\otimes _{h} \mathcal{S}_p^k)$, coming from an arc $ e = (a,b) \in E(D)$ and an arc $ (i,j) \in E(h(a,b))$, the sum of labels is constant and equal to $p(a+b+e-3)+(k+p)$. That is, $p( \hbox{val}(f)-3)+k+p$. Thus, the next result is obtained.
\end{remark}

\begin{lemma}\label{valenceinducedproduct} \cite{LopMunRiu6}
Let $\hat{f}$ be the (super) edge-magic labeling of the graph $D \otimes_h\mathcal{S}_p^k$ induced by a (super) edge-magic labeling $f$ of $D$ (see Remark \ref{remarkspk}). Then the valence of $\hat{f}$ is given by the formula
 \begin{eqnarray}
\hbox{val}(\hat{f}) &=& p(\hbox{val}(f)-3) + k + p.
 \end{eqnarray}

\end{lemma}

To prove the main result, we need some technical lemmas. The next lemma was proved in \cite{LopMunRiu5}.

\begin{lemma}\label{alphap}
\cite{LopMunRiu5} Let $p$ and $q$ be odd coprime numbers. Then there exist integers $\alpha$ and $\beta$ with $1 = \alpha p + \beta q $ and $max \{|\alpha p| , |\beta q|\} \leq (pq+1)/2 $.
\end{lemma}

The following lemma was partially proved in \cite{LopMunRiu5}.

\begin{lemma}\label{xxprime}
Let $p$ and $q$ be different odd primes. Then, there exists an integer $x$ with $1 < x < pq$ such that $\hbox{gcd}(x,pq)\neq 1 $, $\hbox{gcd}(x-1,pq)\neq 1 $. Moreover, if there exists a different $x'$ with $1 < x' < pq$ such that $\hbox{gcd}(x',pq)\neq 1 $, $\hbox{gcd}(x'-1,pq)\neq 1 $, then $x'=pq-x+1$.
\end{lemma}

\begin{proof}
By Lemma \ref{alphap}, there exist two integers $\alpha$ and $\beta$ such that $\hbox{max}\{|\alpha p| , |\beta q|\} \leq (pq+1)/2 $ and $\alpha p + \beta q = 1$. Assume, without loss of restriction that, $\alpha p > 0$. Let $x=\alpha p$. Then we have that $x-1= - \beta q$. Thus, $\hbox{\hbox{gcd}}(x,pq)=p$ and $\hbox{\hbox{gcd}}(x-1,pq)=q$. Let $x'=pq-x+1$. Then we have that $\hbox{\hbox{gcd}}(x',pq)=q$, $\hbox{\hbox{gcd}}(x'-1,pq)=p$. Now, we show that $1 < x,x' < pq$. Since $\alpha $ and $p$ are positive integers, $x= \alpha p > 1$. Using Lemma \ref{alphap}, we have $x=\alpha p \leq (pq+1)/2 < pq.$ Thus, we obtain that $1<x'=pq-\alpha p+1<pq$. Hence, $1 < x,x' < pq$.

Finally, we prove that $x$ and $x'$ are unique. Suppose that there exists another $y$ such that $1 < y < pq$ with $\hbox{\hbox{gcd}}(y,pq)\neq 1$ and $\hbox{\hbox{gcd}}(y-1,pq)\neq 1$. By considering $y'=pq-y+1$, we can assume that $\hbox{\hbox{gcd}}(y,pq)=p$ and $\hbox{\hbox{gcd}}(y-1,pq) =q$. Let $\alpha '$ and $\beta '$ be such that $y= \alpha ' p$ and $ y-1= \beta ' q$. Then, $1 \leq \alpha ' < q$ and $ 1 \leq \beta ' < p$. Hence, $|y-x|=| \alpha ' - \alpha |p < pq$ and $|y-x|=|y-1-(x-1)|=| \beta ' + \beta |q$. However, $| \beta ' + \beta |q =| \alpha ' - \alpha | p$, a contradiction since $p$ and $q$ are different primes. Therefore, $x=y$.
\end{proof}
\begin{corollary}\label{coro: els valors de r conflictius}
Let $p$ and $q$ be different odd primes. Then, there exist exactly $2p^{k-1}$ integers $y$ with $1 < y < p^kq$ such that $\hbox{gcd}(y,p^kq)\neq 1 $, $\hbox{gcd}(y-1,p^kq)\neq 1 $. Moreover, these integers are of the form $x+\lambda pq$, $x'+\lambda pq\in [1,p^kq]$, where  $\lambda$ is a integer in $ [0,p^{k-1}-1]$ and $x,x'$ are the numbers described in Lemma \ref{xxprime}.
\end{corollary}

\begin{proof}
Let $x$ and $x'$ be the integers described in Lemma \ref{xxprime}. Then, for every integer $\lambda\in [0,p^{k-1}-1]$, we get $x+\lambda pq$, $x'+\lambda pq\in [1,p^kq]$ and $\hbox{gcd}(x,p^kq)\neq 1 $, $\hbox{gcd}(x-1,p^kq)\neq 1 $. Similarly, for every $y\in [1,p^kq]$ with $\hbox{gcd}(y,p^kq)\neq 1 $, $\hbox{gcd}(y-1,p^kq)\neq 1 $, there exists a positive integer $\lambda\in [0,p^{k-1}-1]$ such that $\lambda pq<y<(\lambda+1)pq$. Thus, $y-\lambda pq\in [1,pq]$, with $\hbox{gcd}(y,pq)\neq 1 $, $\hbox{gcd}(y-1,pq)\neq 1 $, since $p$ and $q$ are different primes. Hence, $y-\lambda pq$ is one of the two possible integers described in Lemma \ref{xxprime}.
\end{proof}

\section{A family of perfect edge-magic graphs of the form $C_{m} \odot \overline K_{n}$} \label{section: main}
Let $L$ be the set of vertices of degree $1$ of $G=C_{m} \odot \overline K_{n}$ and $C=V(G)\setminus L$. Assume that $C= \{v_{0},v_{1},\ldots,v_{m-1}\}$, $L= \{ v_{i}^{j}\}_{i=0,1,2,\ldots,m-1}^{ j = 1,2,\ldots,n } $ and $E(G)= \{ v_{i}v_{i+_{m}1},v_{i}v_{i}^{j}\}_{i=0,1,2,\ldots,m-1}^{ j = 1,2,\ldots,n}$ where $+_{m}$ denotes the sum modulo $m$. Let $ \overrightarrow G$ be an orientation of $G$ such that, the subdigraph induced by $C$ is strongly connected and all vertices of degree $1$ have indegree $1$. Note that, $\overrightarrow {G} \cong \overrightarrow K_{1,n}^{l} \otimes  C^+_{m}$, where $\overrightarrow K_{1,n}^{l}$ is the digraph obtained by orienting $K_{1,n}^l$ in such a way that all vertices of degree $1$ have indegree $1$ and $  C^+_{m}$ is a strong orientation of $C_{m}$. 

The following construction and lemmas are inspirated by the construction introduced by L\'opez et al. in \cite{LopMunRiu5}. Let $\mathcal{ M}_m$ be the set of all matrices of order $m \times m$ and let $g_{1}$ be the labeling of $\overrightarrow G$ induced by the product $ \overrightarrow K_{1,n}^{l} \otimes  C^+_{m}$, when considering the super edge-magic labeling of  $ \overrightarrow K_{1,n}^{l} $ that assigns label $1$ to the central vertex and a super edge-magic labeling $g$ of $C_{m}$. By identifying each vertex of $\overrightarrow G$ with the label assigned to it by $g_{1}$, we can construct the adjacency matrix of the digraph $\overrightarrow G$, which is of the form: $A_g^{1}=(A_{ij}^{1})$, where each $A_{ij}^{1} \in \mathcal{M}_m, A_{ij}^{1}=0$ for $i>1$ and $A_{1j}^{1}$ has the structure of the adjacency matrix of $g(C_{m})^+$, in which each vertex of $C_m^+$ is identified with the label assigned to it by $g$.
We can also consider the opposite strong orientation of the labeled cycle denoted by $ g(C_{m})^-$. If we identify each vertex of $\overrightarrow G\cong\overrightarrow K_{1,n}^{l} \otimes  C^-_{m}$ with the labels induced by the product, we obtain an adjacency matrix of $\overrightarrow G$ with the same structure as $A_g^{1}$. Let us denote this matrix by $B_g^{1}$. Then $B_g^{1}=(B_{ij}^{1})$, where each $B_{ij}^{1} \in \mathcal{M}_m, B_{ij}^{1}=0$ for $i>1$ and $B_{1j}^{1}$ has the structure of the adjacency matrix of $g(C_{m})^-$, in which each vertex of $C_m^-$ is identified with the label assigned to it by $g$.

Let $A_g^{r}$ and $B_g^{r}$ be the matrices obtained from $A_g^{1}$ and $B_g^{1}$ respectively by translating each row $r-1$ units, for $1\le r \leq mn+1$. Thus, if $A_g^r = (a_{ij}^{r})$, then
\begin{equation}\label{formula: translated matrix}
   a_{ij}^r=\left\{\begin{array}{ll}
a_{(i-r+1)j}^1, & i \ge r \\
 0, & \hbox{otherwise}. \end{array}\right.
\end{equation}

 Let $G(A_g^{r})$ and $ G(B_g^{r})$ be the digraphs with adjacency matrices $A_g^{r}$ and $B_g^{r}$ respectively. We also denote by $S(A_g^{r})$ and $S(B_g^{r})$ the subdigraphs of $G(A_g^{r})$ and $ G(B_g^{r})$ induced by the set of vertices $\{r,\ldots,r-1+m\}$, respectively. From the adjacency matrices $A_g^{r}$ and $ B_g^r$, it is easy to check the following lemma.

 \begin{lemma}\label{lemma: A^r and B^r, val(f^r)}
Let $g$ be a super edge-magic labeling of $C_m$. The vertices of $G(A_g^{r})$ and $G(B_g^{r})$ define a super edge-magic labeling $g^+_r$ and $g^-_r$, respectively, with valence  $\hbox{val}(g^+_r)=\hbox{val}(g^-_r)= \hbox{val}(g_1)+r-1$, $1\le r\le mn+1$.

 The digraphs $S(A_g^{r})$ and $S(B_g^{r})$ are 1-regular and the graphs $\hbox{und}(G(A_g^{r}))$ and $\hbox{und}(G(B_g^{r}))$ are of the form $H^r_g\odot \overline K_{n}$ where $H^r_g$ is a 2-regular graph. Moreover, $H^r_g\cong H^{r+\lambda m}_g$, for every positive integer $\lambda$ with $r+\lambda m\le (m+1)n$.
 \end{lemma}

\proof The first part of the lemma comes from Lemma \ref{CnSEM}, since the minimum induced sum of two adjacent vertices increases by one unit at every step of the translation, and for $r=1$ this minimum sum is the minimum sum of adjacent vertices of a super edge-magic labeled cycle. The second part is due to the structure of the adjacency matrices.\qed

\begin{lemma}\label{lemma: S(A_f^r) is a cycle}
Let $f$ be the canonical super edge-magic labeling of $C_m$. If neither $\hbox{und}(S(A_f^r))$ nor $\hbox{und}(S(B_f^r))$ is isomorphic to a cycle, then $\hbox{\hbox{gcd}}((m+1)/2-(r-1),m)\neq 1$ and $\hbox{\hbox{gcd}}(m-1)/2-(r-1),m) \neq 1$.
\end{lemma}

\proof By (\ref{formula: translated matrix}), it is clear that $(a,b) \in E(S(A_f^r))$ if and only if $(a- (r-1),b) \in E(G(A_f^1))$. That is, if and only if $(b-a) \equiv (m+1)/2 - (r-1) \ (\hbox{mod} \ m)$, by (\ref{conditionforcycle1}) in Remark \ref{conditionforcycle1}. Similarly, $(a,b) \in E(G(B_f^r))$, if and only if  $(b-a) \equiv (m-1)/2 - (r-1) \ (\hbox{mod} \ m)$, by (\ref{conditionforcycle2}). Hence, the result follows.\qed

\begin{theorem}\label{pqsem}
Let $m=pq$ where $p$ and $q$ are different odd primes. Let $n$ be a positive integer. Then, the graph $G = C_{m} \odot \overline K_{n}$ is perfect super edge-magic.
\end{theorem}

\begin{proof}
Let us first determine the super edge-magic interval $I_G$ of $G$. The maximum of $I_{G}$ occurs when
$\{g(u):u \in L\} = \{1,2,3,\ldots,mn\}$ and the minimum when $\{g(u):u \in L \} = \{m+1, m+2,\ldots,m+mn\}$ where $L$ denotes the set of vertices of degree $1$ of $G$ and $g:V(G) \rightarrow \{ i \}_{i=1}^{m+mn}$ is any bijective function. Thus, $I_G=[(3+5m)/2+2mn,(3+5m)/2+3mn]\cap \mathbb{Z}$.


Let $f$ be the canonical labeling of the cycle. By Lemmas \ref{lemma: A^r and B^r, val(f^r)} and \ref{lemma: S(A_f^r) is a cycle}, we obtain that for all $r$ with $1 \leq r \leq mn+1$, with either $\hbox{\hbox{gcd}}((m+1)/2-(r-1),m) = 1$ or $\hbox{\hbox{gcd}}((m-1)/2-(r-1),m) = 1$,  either $A_f^r$ or $B_f^r$ is the adjacency matrix of a super edge-magic labeled digraph, whose underlying graph is $G$. Moreover, if $f_r$ is the induced super edge-magic labeling of $G$, then $\hbox{val}(f_r)= \hbox{val}(f_1)+r-1$. Notice that, by Lemma \ref{valenceinducedproduct}, val$(f_1)=(5m+3)/2+2mn$.

Now, we provide a construction to cover the missing valences of $G$. That is, $\hbox{val}(f_1)+r-1$, with $\hbox{\hbox{gcd}}((m+1)/2-(r-1),m)\neq 1$ and $\hbox{\hbox{gcd}}(m-1)/2-(r-1),m) \neq 1$. What happens for this values is that, by Lemma \ref{lemma: S(A_f^r) is a cycle}, we can not guarantee that $H^r_f$ is a cycle. Also, by Lemma \ref{lemma: S(A_f^r) is a cycle}, we have that $H^r_f\cong H^{r+\lambda m}_f$, for every positive integer $\lambda$ with $r+\lambda m\le (m+1)n$. Thus, in what follows, we will assume that $n=1$.

Let $\alpha p+\beta q=1$ be the B\'ezout identity where $\alpha p>0$ and $\max \{\alpha p, |\beta q|\}\le (pq+1)/2$ (such $\alpha$ and $\beta$ exist by Lemma \ref{alphap}). Then, $x=\alpha p$ is one of the integers of Lemma \ref{xxprime}. The other one is $x'=pq-\alpha p+1$. Thus, one of the missing valences is  $\hbox{val}(f_1)+r-1$, where $r-1=(pq+1)/2-\alpha p$. That is,
\begin{eqnarray}
 \nonumber 
  r-1 &=&p(\frac{q+1}{2}-\alpha-1)+\frac{p+1}{2}.
\end{eqnarray}

Let $\alpha'= (q+1)/{2}-\alpha$ and $\beta'= (p+3)/2$. Then, $r-1=p(\alpha'-1)+\beta'-1$.

Notice that, if we prove the existence of a super edge-magic labeling $g_r$ of $G$ with valence $\hbox{val}(f_1)+r-1$, then the other missing valence, namely, $\hbox{val}(f_1)+(pq-1)/2+\alpha p$ will be realized by the super edge-magic complementary labeling of $g_r$, namely $g_r^c$ (see  (\ref{formula: valence of the super complentary})).

Let $g$ be the labeling of $C_m^+$ induced by the product $ f(C_{q})^+ \otimes  f(C_{p})^-$, when considering the canonical super edge-magic labeling of  $C_{q}$ and $C_{p}$, respectively. We will prove that $H^r_g\cong \hbox{und}(S(A_g^{r}))$ is a cycle of length $pq$.

Let $(a', b') \in E(S(A_g^{r}))$, that is $r \leq a', b' \leq r-1+m$. Thus, $(a' -(r-1), b') \in E(G(A_g^{1}))$. In particular, there exists a nonnegative integer $\lambda_0= \lambda_0( b' )$ such that $(a' -(r-1), b'- \lambda_0m) \in E(S(A_g^{1}))$. Let $(a,i),(b,j)$ be such that $a' = p(a-1)+i$, $b' - \lambda_0m=p(b-1)+j$ where $1 \leq a \leq 2q$, $1 \leq b \leq q$ and $1 \leq i,j \leq p$. This implies, $(p(a-\alpha')+(i- \beta'+1), p(b-1)+j) \in E(S(A_g^{1}))$. That is, $(p(a-\alpha')+(i- \beta'+1), p(b-1)+j) \in E(  f(C_{q})^+ \otimes  f(C_{p})^-)$.

We have two types of adjacencies:

\textbf{Type i:} $1 \leq i-\beta'+1 \leq p$.
By definition of $\otimes$-product and the labeling induced (see Remark \ref{remarkspk}), we obtain that $(a-\alpha'+1,b) \in E(f(C_{q})^+)$ and $(i-\beta'+1,j) \in E( f(C_{p})^-)$. That is, using (\ref{conditionforcycle1}) and (\ref{conditionforcycle2}) in Remark \ref{strongorientationofcm},
 $b-(a-\alpha'+1)\equiv(q+1)/2 \ (\hbox{mod} \ q)$ and $ j-(i- \beta'+1) \equiv (p-1)/2 \ (\hbox{mod} \ p).$ Equivalently,
 $b-a\equiv \alpha +1\ (\hbox{mod} \ q)$ and $j-i\equiv -1 \ (\hbox{mod} \ p)$.

\textbf{Type ii:} $-p+2 \leq i- \beta'+1 \leq 0$.
Again by definition of $\otimes$-product and the labeling induced, we obtain  $(a-\alpha,b) \in E(f(C_{q})^+)$ and $(p+i-\beta'+1,j) \in E( f(C_{p})^-)$. Thus, using (\ref{conditionforcycle1}) and (\ref{conditionforcycle2}), $b-(a-\alpha') \equiv (q+1)/2 \ (\hbox{mod} \ q)$ and $j-(p+i- \beta'+1)\equiv (p-1)/2 \ (\hbox{mod} \ p)$.  Equivalently,
 $b-a\equiv \alpha \ (\hbox{mod} \ q)$ and $j-i\equiv -1 \ (\hbox{mod} \ p)$.

Assume that $r$ is contained in a cycle ${C}^+_l$, $l<pq$, with $I$ edges of type $i$. Then,
$l=kp$, $I=k(p-1)/2$, for some positive integer $k$, and
\begin{eqnarray} \label{eq_adjacency of C_l}
  kp\alpha +k(p-1)/2 &=& sq,
\end{eqnarray}

for some integer $s$. Using that $\alpha p=1-\beta q$ and (\ref{eq_adjacency of C_l}), we obtain that $q$ is a divisor of $(p+1)/2$. Note that, this implies that $p+1=\lambda q$, for some positive $\lambda$, and hence, $p(\alpha +1)=(\lambda -\beta)q$. Therefore, $q$ divides $\alpha+1$ contradicting that $\alpha p<(pq+1)/2$.

\end{proof}

The magic interval of crowns of the form $C_m\odot\overline{K}_n$ was obtained in \cite{PEM_LMR}.

\begin{lemma} \label{magicinterval}
\cite{PEM_LMR} Let $m$ and $n$ be positive integers with $m \geq 3$. Then, the magic interval of $C_m\odot\overline{K}_n$ is given by
$$J_{C_m\odot\overline{K}_n}=\left[\frac{3+5m}2+2mn,\frac{3+7m}{2}+4mn\right]\cap \mathbb{Z}.$$
\end{lemma}

Theorem \ref{pqsem} implies that for every element $k$ included in the super edge-magic interval, there exists a super edge-magic labeling with valence $k$. Taking the complementary labeling of these labelings, we get that all natural numbers from $3mn+(3+7m)/2$ up to $4mn+(3+7m)/2$ appear as valences of edge-magic labelings of $C_m \odot \overline K_{n} $. Therefore, in order to prove that $C_{m} \odot \overline K_{n}$ is perfect edge-magic, we only need to show that for each $k \in \mathbb{N}$, with $3mn+(3+5m)/2 < k < 3mn+(3+7m)/2$, there exists an edge-magic labeling with valence $k$. We do this using the odd and even labelings of the labelings $f_{r}$ and $g_r$ introduced in the proof of Theorem \ref{pqsem}.

\begin{lemma}
Let $m$ be the product of two different odd primes and let $n$ be any positive integer. Then, for each $k$ with $2mn+3m+1 \leq k \leq 4mn+3m+2$ there exists an edge-magic labeling of $C_{m} \odot \overline K_{n}$ with valence $k$.
\end{lemma}

\begin{proof}
Let $m=pq$, where $p$ and $q$ are different odd primes. Let $x$ and $x'$ be the integers introduced in Lemma \ref{xxprime}. Consider the super edge-magic labelings $f_r$ and $g_r$ of $C_{m} \odot \overline K_{n}$, introduced in the proof of Theorem \ref{pqsem}. Then, the set $\{ \hbox{val}(f_{r}) ; 1 \leq r \leq mn+1, r-1\notin \{(pq+1)-x,(pq+1)-x'\}\}\cup \{ \hbox{val}(g_{r}),\hbox{val}(g^c_{r}):\ r-1= (pq+1)-x\}$ is a set of consecutive integers. Thus, Lemma \ref{lemma: oddevenlabelings} implies that the set $ \{ \hbox{val}(o(f_{r})), \hbox{val}(e(f_{r})) ; 1 \leq r \leq mn+1, r-1\notin \{(pq+1)-x,(pq+1)-x'\}\}\cup \{ \hbox{val}(o(g_{r})),\hbox{val}(o(g^c_{r})), \hbox{val}(e(g_{r})),\hbox{val}(e(g^c_{r})):\ r-1= (pq+1)-x\}$ contains all integers from $\hbox{val}(o(f_{1}))$ up to $\hbox{val}(e(f_{mn+1}))$. That is, all integers from $2mn+3m+1$ up to $4mn+3m+2$.
\end{proof}

Since $2mn+3m+1 \leq 3mn+(3+5m)/2$ and $3mn+(3+7m)/2 \leq 4mn+3m+2$ for $n \geq 1$, we obtain the next theorem.

\begin{theorem}
Let $m=pq$ where $p$ and $q$ are different odd primes. Let $n$ be a positive integer. Then, the graph $G = C_{m} \odot \overline K_{n}$ is perfect edge-magic.
\end{theorem}

\section{Edge-magic labelings of crowns}

The fact that even cycles admit edge-magic labelings has been known for several decades already. See the next theorem.

\begin{theorem} \label{evencyclevalence}
\cite{Wa}Every even cycle $C_n$ has an edge-magic labeling with magic sum $(5n+4)/2$.
\end{theorem}

In fact this result has been improved recently as shown in the next theorem. It is also worth to mention that McQuillian \cite{McQ09} has made important contributions in this direction.
\begin{theorem}\cite{PEM_LMR2}\label{evencnvalencepast}
Let $m=2^{\alpha}p_1^{\alpha_1}p_2^{\alpha_2}\ldots p_k^{\alpha_k}$ be the unique prime factorization (up to ordering) of an even number $m$. Then $C_m$ admits at least $\Sigma_{i=1}^{k}{\alpha_i}$ edge-magic labelings with at least $\Sigma_{i=1}^{k}{\alpha_i}$ mutually different magic sums. If $ \alpha \geq 2$, this lower bound can be improved to $1+\Sigma_{i=1}^{k}{\alpha_i}$.
\end{theorem}

Similarly, the next result was established in \cite{PEM_LMR2} for cycles of odd order.

\begin{theorem}\cite{PEM_LMR2}\label{oddcnvalencepast}
Let $m=p_1^{\alpha_1}p_2^{\alpha_2}\ldots p_k^{\alpha_k}$ be the unique prime factorization (up to ordering) of an odd number $m$. Then $C_m$ admits at least $1+\Sigma_{i=1}^{k}{\alpha_i}$ edge-magic labelings with at least $1+\Sigma_{i=1}^{k}{\alpha_i}$ mutually different magic sums.
\end{theorem}

\begin{lemma}\label{repeatedvalences}
Let $g:V(C_m^+) \cup E(C_m^+) \rightarrow \{1,2,\ldots,2m\}$ be an edge-magic labeling of $C_m^+$, and let $\gamma_r:V(\overrightarrow{K}_{1,n}^l) \rightarrow \{1,2,\ldots,n+1\}$ be a super edge-magic labeling of $\overrightarrow{K}_{1,n}^l$ that assigns label $r$ to the central vertex with $\hbox{val}(\gamma_r)=r+2n+3, \ 1 \leq r \leq n+1$. Then the induced edge-magic labeling $\widehat{g}_r$ of $C^+_m \otimes_h \overrightarrow{K}_{1,n}^l$ has valence $(n+1)(\hbox{val}(g)-2)+r+1$. Let $g'$ be a different edge-magic labeling of $C^+_m$ with $\hbox{val}(g) < \hbox{val}(g')$, then $\hbox{val}(\widehat{g}_{n+1}) < \hbox{val}(\widehat{g}_{1}')$, where $\widehat{g}_r'$ is the induced edge-magic labeling of $C^+_m \otimes \overrightarrow{K}_{1,n}^l$ when $\overrightarrow{K}_{1,n}^l$ is labeled with $\gamma_r$ and $C^+_m$ with $g'$. \end{lemma}

\begin{proof}
By Lemma \ref{valenceinducedproduct}, $\hbox{val}(\widehat{g}_{r})=(n+1)[\hbox{val}(g)-3]+r+1+n+1$, that is,
$\hbox{val}(\widehat{g}_{r})=(n+1)[\hbox{val}(g)-2]+r+1$. Let $g'$ be a different edge-magic labeling of $C^+_m$ with $\hbox{val}(g) < \hbox{val}(g')$, then  $\hbox{val}(\widehat{g}_{n+1})= (n+1)[\hbox{val}(g)-2]+n+2 \leq (n+1)[\hbox{val}(g')-1-2]+n+2 <  \hbox{val}(\widehat{g}_{1}')$. Hence the result follows.
\end{proof}
Now using Theorem \ref{evencnvalencepast} and Lemmas \ref{k1nsem} and \ref{repeatedvalences}, we can prove the next theorem.

\begin{theorem}
Let $m=2^{\alpha}p_1^{\alpha_1}p_2^{\alpha_2}\ldots p_k^{\alpha_k}$ be the unique prime factorization (up to ordering) of an even number $m$. Then $G=C_m \odot \overline{K}_n$ admits at least $(\Sigma_{i=1}^{k}{\alpha_i})(n+1)$ mutually different magic sums. If $ \alpha \geq 2$, this lower bound can be improved to $(1 +  \Sigma_{i=1}^{k}{\alpha_i})(n+1)$.
\end{theorem}

\begin{proof}
 Note that $G \cong \hbox{und}(C_m^+ \otimes \overrightarrow K_{1,n}^{l})$. Let $\widehat{g}$ and $g$ be edge-magic labelings of $C_m^+ \otimes \overrightarrow K_{1,n}^{l}$ and $C_m^+$ respectively and let $\gamma_r$ be a super edge-magic labeling of $\overrightarrow{K}_{1,n}^l$ that assigns label $r$ to the central vertex with $\hbox{val}(\gamma_r)=r+2n+3$, $1 \leq r \leq n+1$. By Lemma \ref{repeatedvalences}, we get $\hbox{val}(\widehat{g_r})=(n+1)[\hbox{val}(g)-2]+r+1$. Thus, $\hbox{val}(\widehat{g})$ depends on the valences of $g$ and $r$. We know that by Lemma \ref{k1nsem}, $\overrightarrow{K}_{1,n}^l$ has $n+1$ valences and by Theorem \ref{evencnvalencepast}, $C_m$ has at least $\Sigma_{i=1}^{k}{\alpha_i}$ mutually different valences. Thus, using Lemma \ref{repeatedvalences}, $G=C_m \odot \overline{K}_n$ admits at least $(\Sigma_{i=1}^{k}{\alpha_i})(n+1)$ mutually different magic sums. If $ \alpha \geq 2$, this lower bound can be improved to $(1 + \Sigma_{i=1}^{k}{\alpha_i})(n+1)$.
\end{proof}

Similarly, using Theorem \ref{oddcnvalencepast} and Lemmas \ref{k1nsem} and \ref{repeatedvalences}, we can prove the next theorem.

\begin{theorem}
Let $m=p_1^{\alpha_1}p_2^{\alpha_2}\ldots p_k^{\alpha_k}$ be the unique prime factorization (up to ordering) of an odd number $m$. Then $G=C_m \odot \overline{K}_n$ admits at least $(1 +  \Sigma_{i=1}^{k}{\alpha_i})(n+1)$ mutually different magic sums.
\end{theorem}

Let $f$ be the canonical labeling of the cycle $C_{p^kq}$, where $p$ and $q$ are different odd primes and $k$ is a positive integer. The construction provide in Section \ref{section: main} guarantees the existence of a super edge-magic labeling of the crown $C_{p^kq}\odot \bar K_n$, with valence $ \hbox{val}(f_1)+r-1$, for many values of $r$. The possible exceptions can be obtained from Corollary \ref{coro: els valors de r conflictius}.

\begin{openQ}
Prove or disprove that $C_{p^kq}\odot \overline K_n$, where $p$ and $q$ are different odd primes and $k$ is a positive integer is perfect (super) edge-magic.
\end{openQ}

\noindent {\bf Acknowledgements}
The research conducted in this document by the first author has been supported by the Spanish Research Council under project
MTM2011-28800-C02-01 and symbolically by the Catalan Research Council
under grant 2014SGR1147.

\end{document}